\documentclass[11pt]{article}
\usepackage{amsmath,amsfonts,amssymb, amsthm,enumerate,graphicx}
\usepackage{tikz,authblk}
\usepackage{subfig}
\usepackage{url}
\usepackage{stackengine,scalerel}
\newcommand\ttimes{\mathbin{\ThisStyle{\ensurestackMath{%
  \stackengine{-1\LMpt}{\SavedStyle\times}
  {\SavedStyle_{\hstretch{.9}{\mkern1mu\sim}}}{O}{c}{F}{T}{S}}}}}
\newtheorem{theorem} {{\textsf{Theorem}}}
\newtheorem{proposition}[theorem]{{\textsf{Proposition}}}

\newtheorem{definition}[theorem]{{\textsf{Definition}}}
\newtheorem{remark}[theorem]{{\textsf{Remark}}}

\newtheorem{lemma}[theorem]{{\textsf{Lemma}}}

\newcommand{\xdownarrow}[1]{%
  {\left\downarrow\vbox to #1{}\right.\kern-\nulldelimiterspace}
}
\begin{document} 

\title{Handle decompositions for a class of closed orientable PL 4-manifolds}
\author{Biplab Basak and Manisha Binjola}

\date{}

\maketitle

\vspace{-15mm}
\begin{center}

\noindent {\small Department of Mathematics, Indian Institute of Technology Delhi, New Delhi 110016, India.$^1$}


\footnotetext[1]{{\em E-mail addresses:} \url{biplab@iitd.ac.in} (B.
Basak), \url{binjolamanisha@gmail.com} (M. Binjola).}

\medskip

\date{June 22, 2023}
\end{center}

\hrule

\begin{abstract}
In this article, we study a class of closed connected orientable PL $4$-manifolds admitting a semi-simple crystallization and which have an infinite cyclic fundamental group. We show that the manifold in the class admits a handle decomposition in which the number of $2$-handles depends upon its second Betti number and other $h$-handles ($h \leq 4$) are at most $2$. 
More precisely, our main result is the following. For a closed connected orientable PL $4$-manifold having a semi-simple crystallization with the  fundamental group as $\mathbb{{Z}}$, we have constructed a handle decomposition for $M$ as one of the following types:
\begin{enumerate}[$(1)$]
\item one $0$-handle, two $1$-handles, $1+\beta_2(M)$ $2$-handles, one $3$-handle and one $4$-handle,
\item one $0$-handle, one $1$-handle, $\beta_2(M)$ $2$-handles, one $3$-handle and one $4$-handle,
\end{enumerate}
 where $\beta_2(M)$ denotes the second Betti number of manifold $M$ with $\mathbb{Z}$ coefficients.
\end{abstract}
 
\noindent {\small {\em MSC 2020\,:} Primary 57Q15; Secondary  05C15, 05E45, 57Q05, 57M15, 57M50.

}

\medskip

\section{Introduction}

A crystallization $(\Gamma,\gamma)$ of a closed connected PL $d$-manifold is a certain type of edge colored graph which represents the manifold (details provided in Subsection \ref{crystal}). The journey of crystallization theory has begun due to Pezzana who gives the existence of a crystallization  for every closed connected PL $d$-manifold (see \cite{pe74}). 

\smallskip

Extending the notion of genus in $2$ dimension, the term regular genus for a closed connected PL $d$-manifold has been introduced in \cite{g81}, which is related to the existence of regular embeddings of graphs representing the manifold into surfaces (cf. Subsection \ref{sec:genus} for details). 

\smallskip

  In \cite{bc17}, the term semi-simple gems is introduced for closed $4$-manifolds. In this paper, we particularly work on closed connected PL $4$-manifolds admitting semi-simple crystallizations.
 
 \smallskip
 
  A problem for closed $4$-manifolds was posed by Kirby and which can be formulated as : ``Does every simply connected closed $4$-manifold have a handlebody decomposition without $1$-handles?" Many researchers worked on it for the decades, in manifold with boundary as well.
  
  \smallskip
 
  In this paper, we extend the earlier known work to the closed connected $4$-manifolds with the  fundamental group $\mathbb{Z}$ and precisely take a large class of manifolds admitting semi-simple crystallizations. Also, the class of PL $4$-manifolds admitting semi-simple crystallizations is not completely known by now. Recently in \cite{cc20}, the authors gave a class of compact $4$-manifolds with empty or connected boundary which admit a special handle decomposition lacking in $1$-handles and $3$-handles. In this article, we show the following:  for a closed connected orientable PL $4$-manifold having a semi-simple crystallization with fundamental group as $\mathbb{{Z}}$, we construct a handle decomposition for $M$ as one of the following types:
  \begin{enumerate}[$(1)$]
\item  one $0$-handle, two $1$-handles, $1+\beta_2(M)$ $2$-handles, one $3$-handle and one $4$-handle,
\item one $0$-handle, one $1$-handle, $\beta_2(M)$ $2$-handles, one $3$-handle and one $4$-handle.
\end{enumerate}

\section{Preliminaries}
Crystallization theory  provides a combinatorial tool for representing piecewise-linear (PL) manifolds of arbitrary dimension via colored graphs and is used to study geometrical and topological properties of manifolds.

\subsection{Crystallization} \label{crystal}

For a multigraph $\Gamma= (V(\Gamma),E(\Gamma))$ without loops, a surjective map $\gamma : E(\Gamma) \to \Delta_d:=\{0,1, \dots , d\}$ is called a proper edge-coloring if $\gamma(e) \ne \gamma(f)$ for any two adjacent edges $e$ and $f$. The elements of the set $\Delta_d$ are called the {\it colors} of $\Gamma$. A graph $(\Gamma,\gamma)$ is called {\it $(d+1)$-regular} if degree of each vertex is $d+1$.
 We refer to \cite{bm08} for standard terminology on graphs. All spaces and maps will be considered in PL-category.

\smallskip

A  regular {\it $(d+1)$-colored graph} is a pair $(\Gamma,\gamma)$, where $\Gamma$ is $(d+1)$-regular and $\gamma$ is a proper edge-coloring. 
If there is no confusion with coloration, one can use $\Gamma$ for $(d+1)$-colored graphs instead of $(\Gamma,\gamma)$. For each $B \subseteq \Delta_d$ with $h$ elements, the graph $\Gamma_B =(V(\Gamma), \gamma^{-1}(B))$ is an $h$-colored graph with edge-coloring $\gamma|_{\gamma^{-1}(B)}$. For a color set $\{j_1,j_2,\dots,j_k\} \subset \Delta_d$, $g(\Gamma_{\{j_1,j_2, \dots, j_k\}})$ or $g_{j_1j_2 \dots j_k}$ denotes the number of connected components of the graph $\Gamma_{\{j_1, j_2, \dots, j_k\}}$. 
. A graph $(\Gamma,\gamma)$ is called {\it contracted} if subgraph $\Gamma_{\hat{c}}:=\Gamma_{\Delta_d\setminus c}$ is connected for all $c$.

\smallskip
 
Let $\mathbb{G}_d$ denote the set of regular $(d+1)$-colored graphs $(\Gamma,\gamma)$.  For each  $(\Gamma,\gamma) \in \mathbb{G}_d$, a corresponding $d$-dimensional simplicial cell-complex ${\mathcal K}(\Gamma)$ is determined as follows:

\begin{itemize}
\item{} for each vertex $u\in V(\Gamma)$, take a $d$-simplex $\sigma(u)$ and label its vertices by $\Delta_d$;

\item{} corresponding to each edge of color $j$ between $u,v\in V(\Gamma)$, identify the ($d-1$)-faces of $\sigma(u)$ and $\sigma(v)$ opposite to $j$-labeled vertices such that the vertices with same label coincide.
\end{itemize}

 The geometric carrier $|\mathcal{K}(\Gamma)|$ is a $d$-pseudomanifold and $(\Gamma,\gamma)$ is said to be a gem (graph encoded manifold) of any $d$-pseudomanifold homeomorphic to $|\mathcal{K}(\Gamma)|$ or simply is said to represent the $d$-pseudomanifold. We refer to \cite{bj84} for CW-complexes and related notions. It is known via the construction that for $\mathcal{B} \subset \Delta_d$ of cardinality $h+1$, ${\mathcal K}(\Gamma)$ has as many $h$-simplices with vertices labeled by $\mathcal{B}$ as many connected components of $\Gamma_{\Delta_d \setminus \mathcal{B}}$ are (cf. \cite{fgg86}).

 \smallskip
For a $k$-simplex $\lambda$ of ${\mathcal K}(\Gamma)$, $0 \leq k \leq d$, the  star of $\lambda$ in ${\mathcal K}(\Gamma)$ is the pseudocomplex obtained by
taking the $d$-simplices of ${\mathcal K}(\Gamma)$ which contain $\lambda$ and identifying only their
$(d-1)$-faces containing $\lambda$ as per gluings in ${\mathcal K}(\Gamma)$. The link of $\lambda$ in ${\mathcal K}(\Gamma)$ is the subcomplex of its star obtained
by the simplices that do not contain $\lambda$.
It is known that the $|{\mathcal K}(\Gamma_{\hat{c}})|$ is homeomorphic to the link of vertex $c$ of ${\mathcal K}(\Gamma)$ in the first barycentric subdivision of ${\mathcal K}(\Gamma)$. And from the correspondence between $(d+1)$-regular colored graphs and $d$-pseudomanifolds, we have that $|\mathcal{K}(\Gamma)|$ is a closed connected PL $d$-manifold if and only if for each $c\in \Delta_d$, $\Gamma_{\hat{c}}$ represents $\mathbb{S}^{d-1}$.

 \begin{definition}
A $(d+1)$-colored graph  $(\Gamma,\gamma)$ which is a gem of a closed $d$-manifold $M$ is called a {\it crystallization} of $M$ if it is contracted.\\
In this case, there are exactly $d+1$ number of vertices in the corresponding colored triangulation. 
\end{definition}

The initial point of the crystallization theory is the Pezzana's existence theorem (cf. \cite{pe74}) which gives existence of a crystallization for a closed connected PL $d$-manifold. 

 
 

\subsection{Regular Genus of closed PL $d$-manifolds}\label{sec:genus}

 In \cite{g81}, the author extended the notion of genus to arbitrary dimension as regular genus. Roughly, if $(\Gamma,\gamma)\in \mathbb{G}_d$ is a bipartite (resp. non bipartite) $(d+1)$-regular colored graph which represents a closed connected PL $d$-manifold $M$ then for each cyclic permutation $\varepsilon=(\varepsilon_0,\dots,\varepsilon_d)$ of $\Delta_d$, there exists a regular imbedding of $\Gamma$ into an orientable (resp. non orientable) surface $F_\varepsilon$.
 Moreover, the Euler characteristic $\chi_\varepsilon(\Gamma)$ of $F_ \varepsilon$ satisfies
 $$\chi_\varepsilon(\Gamma)=\sum_{i \in \mathbb{Z}_{d+1}}g_{\varepsilon_i\varepsilon_{i+1}} + (1-d)p.$$ 
and the genus (resp. half of genus) $\rho_ \varepsilon$ of $F_ \varepsilon$ satisfies
$$\rho_ \varepsilon(\Gamma)=1-\frac{\chi_\varepsilon(\Gamma)}{2}$$ 
where $2p$ is the total number of vertices of $\Gamma.$

\smallskip
The regular genus $\rho(\Gamma)$ of $(\Gamma,\gamma)$ is defined as
$$\rho(\Gamma)= \min \{\rho_{\varepsilon}(\Gamma) \ | \  \varepsilon =(\varepsilon_0,\varepsilon_1,\dots,\varepsilon_d )\ \mbox{ is a cyclic permutation of } \ \Delta_d\}.$$

\smallskip
The regular genus of $M$ is defined as 
 $$\mathcal G(M) = \min \{\rho(\Gamma) \ | \  (\Gamma,\gamma)\in \mathbb{G}_d \mbox{ is a crystallization which represents } M\}.$$


\noindent We need the concept of regular genus of graphs $\rho(\Gamma)$  throughout the paper. We will  use the result  by Montesinos (\cite{a79}) which ensures that the $3$-handles (if any) and the $4$-handle are added in a unique way to obtain the closed $4$-manifold. It is known that each closed, orientable PL $4$-manifold $M$ admits a handle
presentation $M= H_0 \cup p H_1 \cup q H_2 \cup r H_3 \cup H_4 $. From \cite{a79}, we get that $M$ is
uniquely deterrmined by the cobordism between $\partial (H_0 \cup p H_1)$ and $\partial(H_0 \cup p H_1 \cup q H_2)$,
defined by the 2-handles $q H_2$. 
\begin{proposition}[\cite{a79}]\label{prop:unique}
Each closed, orientable 4-manifold with handle presentation 
$$M= H^{(0)}\cup \left( \cup _{p}H^{(1)}_{p}  \right) \cup \left( \cup_{q}H^{(2)}_{q} \right)\cup \left( \cup_{r}H^{(3)}_{r} \right) \cup H^{(4)}.$$
 is completely deterimined
by  $H^{(0)}\cup \left( \cup _{p}H^{(1)}_{p}  \right) \cup \left( \cup_{q}H^{(2)}_{q} \right)$.
\end{proposition}

\section{Semi simple crystallizations of closed 4-manifolds}\label{section:closed}

In \cite{bc17}, semi-simple crystallizations of closed $4$-manifolds have been introduced and they are  minimal with respect to regular genus among the graphs representing the same manifold. The notion of semi-simple crystallizations is generalisation of the  simple crystallizations of closed simply-connected $4$-manifolds (see \cite{bs16}).
\begin{definition}
Let $M$ be a closed $4$-manifold. A $5$-colored graph $\Gamma$ representing $M$ is called semi-simple if $g_{ijk}=  m + 1$ $ \forall i, j, k \in \Delta_4$, where $m$ is the rank of fundamental group of $M$. In other words, the $1$-skeleton of the associated colored triangulation contains exactly $m+1$ number of $1$-simplices for each pair of $0$-simplices.
\end{definition}

From \cite{c92}, we have the following result on the number of components of crystallization representing closed $4$-manifolds and a relation between Euler characteristic and regular genus of crystallizations.
\begin{proposition}[\cite{c92}]
Let $M$ be a closed $4$-manifold and $(\Gamma,\gamma)$ be a crystallization of $M$. Then 
\begin{align}
g_{j-1,j+1}&=g_{j-1,j,j+1}+\rho-\rho_{\hat{j}} \mbox{   } \forall j \in \Delta_4,\label{eqn:grho1}\\ 
g_{\hat{j-1}\hat{j+1}}&=1+\rho-\rho_{\hat{j-1}}-\rho_{\hat{j+1}}\mbox{  }\forall j \in \Delta_4,\label{eqn:grho2}
\end{align}
and
\begin{equation}
\chi(M)=2-2\rho+\sum_{i \in \Delta_4} \rho_{\hat{i}},\label{eqn:chi}
\end{equation}
where  $\rho$ and $\rho_{\hat{i}}$ denote the regular genus of $
\Gamma$ and $\Gamma_{\hat{i}}$ respectively, and 
$\chi(M)$ is the Euler characteristic of $M$.
\end{proposition}

\begin{lemma}\label{lemma:ori}
Let $M$ be a closed connected orientable $4$-manifold. Let $(\Gamma,\gamma)$ be a $5$-colored semi-simple crystallization for $M$. Let $\beta_i(M)$ denotes the $i^{th}$ Betti number of manifold $M$ with $\mathbb{Z}$ coefficients.
 Then  $ g_{j-1,j+1} =4m+\beta_2-2\beta_1+1$, $\forall j \in \Delta_4$. 
\end{lemma}

\begin{proof}
Let $(\Gamma,\gamma)$ be a semi-simple crystallization representing $M$. 
From  Equation \eqref{eqn:grho2}, for $j=k, k+2 \mbox{ (mod }5)$, we get $\rho_{\hat{k-1}}=\rho_{\hat{k+3}}$. This is true for each $k \in \Delta_4$ which implies $\rho_{\hat{i}}=\rho_{\hat{0}} \mbox{  } \forall i \in \Delta_4.$ 
Then, by adding all the equations in \eqref{eqn:grho2} for each $j \in \Delta_4$, we have 
$$5m=5\rho-10\rho_{\hat{0}} \Rightarrow \rho=m+2\rho_{\hat{0}}.$$
From Equation \eqref{eqn:chi}, $\chi(M)=2-2\rho+5\rho_{\hat{0}}$. This implies $\chi(M)=2-2m+\rho_{\hat{0}}.$ Further, from Equation \eqref{eqn:grho1} for $j=0$, we have $g_{14}=2m+\rho_{\hat{0}}+1$ And $g_{14}=4m+\chi(M)-1$.
It follows from Poincaré  duality that $\chi(M)=2+\beta_2-2\beta_1$.
This follows the result.
\end{proof}

\bigskip
From now onwards, we particularly take the manifolds admitting semi-simple crystallizations and with fundamental group $\mathbb{Z}$. This implies, $\beta_1=1$, $m=1$ and $g_{ijk}=2$. It follows from Lemma \ref{lemma:ori} that $g_{14}=\beta_2+3$. 

\bigskip
It is known that every closed $4$-manifold $M$ admits a handle decomposition, i.e.,
$$M=H^{(0)}\cup (H^{(1)}_1 \cup \dots \cup H^{(1)}_{d_1} ) \cup (H^{(2)}_1\cup \dots \cup H^{(2)}_{d_2} ) \cup (H^{(3)}_1\cup \dots \cup H^{(3)}_{d_3} ) \cup H^{(4)},$$
where $H^{(0)}= \mathbb{D}^4$ and each $k$-handle $H^{(k)}_i=\mathbb{D}^k \times 
\mathbb{D}^{4-k}$ (for $1 \leq k \leq 4$, $1 \leq i \leq d_k$), is attached with 
a map $f^{(k)}_i:\partial \mathbb{D}^k \times \mathbb{D}^{4-k} \rightarrow 
\partial (H^{(0)}\cup \dots \cup (H^{(k-1)}_1 \cup \dots \cup H^{(k-1)}_{d_{k-1}}))$.

\bigskip
The following arguement is known in Crystallization theory (See  \cite{g81}). Let $(\Gamma,\gamma)$ be a crystallization of a closed PL $4$-manifold $M$ and $\mathcal{K}(\Gamma)$ be the corresponding triangulation with the vertex set $\Delta_4$. If $B \subset \Delta_4$, then $\mathcal{K}(B)$ denotes the subcomplex of $\mathcal{K}(\Gamma)$ generated by the vertices $i \in B$. Let Sd $\mathcal{K}(\Gamma)$ be the first barycentric subdivision of $\mathcal{K}(\Gamma)$ and $F(i,j)$ (resp. $F(i,j,k)$) be the largest subcomplex of Sd $\mathcal{K}(\Gamma)$, disjoint from Sd $\mathcal{K}(i,j)$ $\cup$ Sd $\mathcal{K}(\Delta_4 \setminus \{i,j\})$ (resp. Sd $\mathcal{K}(i,j,k)$ $\cup$ Sd $\mathcal{K}(\Delta_4 \setminus \{i,j,k\})$). Then, $F(i,j)$ (resp. $F(i,j,k)$) partitions $\mathcal{K}(\Gamma)$ into two subcomplexes $N(i,j)$ and $N(\Delta_4 \setminus \{i,j\})$ $\mbox{ (resp. } N(i,j,k)$ and $N(\Delta_4 \setminus \{i,j,k\})$ )  with $F(i,j)$ (resp. $F(i,j,k)$) as common boundary. Moreover, the polyhedron $|F(i,j)|$ (resp. $|F(i,j,k)|$) is a closed $3$-manifold and $|N(i,j)|$ (resp. $|N(i,j,k)|$) is a regular neighbourhood of the subcomplex $|\mathcal{K}(i,j)| \mbox{ (resp. } |\mathcal{K}(i,j,k) )| \mbox{ in } |\mathcal{K}(\Gamma)|$. Thus, $M$ has a decomposition of type $M=N(i,j)\cup_\phi N(\Delta_4 \setminus \{i,j\})
$, where $\phi$ is a boundary identification.


\begin{remark}\label{remark:oriM}
{\rm
Let $M$ be a closed connected orientable $4$-manifold with  fundamental group $\mathbb{Z}$ and which admits semi simple crystallization $\Gamma$. This implies $g_{ijk}=2$    $\forall i,j,k$ in $\Gamma$ and thus number of $\{ij\}$-colored edges in $\mathcal{K}(\Gamma)$ is $2$. Without loss of generality, we write $M=N(1,4) \cup N(0,2,3)$. We denote  $N(1,4)$ and $ N(0,2,3)$ by $V$ and $V^\prime$ respectively. Since number of $\{14\}$-colored edges is $2$, $V$ is either $\mathbb{S}^1 \times \mathbb{B}^3$ or $\mathbb{S}^1 \ttimes \mathbb{B}^3$, where $\mathbb{S}^1 \times \mathbb{B}^3$ and $\mathbb{S}^1 \ttimes \mathbb{B}^3$ denote direct and twisted product of spaces $\mathbb{S}^1$, $\mathbb{B}^3$ respectively.  From Mayer Vietoris exact sequence of the triples $(M,V,V^\prime)$, we have 
$$0 \rightarrow H_4(M)\rightarrow H_3(\partial V)\rightarrow 0.$$
This implies $M$ is orientable if and only if $\partial V$ is orientable. Thus, $V=\mathbb{S}^1 \times \mathbb{B}^3$ and $\partial V=\mathbb{S}^1 \times \mathbb{S}^2$.}
\end{remark}

\begin{lemma}
Let $M$, $V$ and $V^\prime$ be the spaces as in remark \ref{remark:oriM}. Then 
\begin{equation}\label{eqn:beta21}
\beta_2(V^\prime)-\beta_2(M)-\beta_1(V^\prime)+1=0.
\end{equation}
\end{lemma}

\begin{proof}
Since $V^\prime$ collapses onto the $2$-dimensional complex $|\mathcal{K}(0,2,3)|$, the Mayer Vietoris sequence of the triple $(M,V,V^\prime)$ gives the following long exact sequence.
$$0 \longrightarrow H_3(M) \longrightarrow H_2(\partial V)
\longrightarrow H_2(V) \oplus H_2(V^\prime) \longrightarrow H_2(M)\longrightarrow H_1(\partial V)$$
$\vspace{-0.8cm}$

$\hspace{12.5cm}\xdownarrow{0.3cm}$

$\hspace{8.25cm} 0\longleftarrow H_1(M) \longleftarrow H_1(V) \oplus H_1(V^\prime)$\\

\smallskip
By assumption  $\pi_1(M)\cong \mathbb{Z}$ which implies  $H_1(M)\cong \mathbb{Z}$. By Poincaré  duality and Universal Coefficient theorem, $H_3(M)\cong H^1(M) \cong FH_1(M) \cong \mathbb{Z}$. Remark \ref{remark:oriM} gives $V=\mathbb{S}^1 \times \mathbb{B}^3$ and $\partial V=\mathbb{S}^1 \times \mathbb{S}^2$.  Now, $H_2(\partial V)\cong H_1(\partial V)\cong \mathbb{Z}$, $H_2(V)\cong 0$ and  $H_1(V)\cong Z.$ Thus above exact sequence reduces to
$$0 \longrightarrow \mathbb{Z} \longrightarrow \mathbb{Z} \longrightarrow H_2(V^\prime) \longrightarrow H_2(M) \longrightarrow \mathbb{Z}\longrightarrow \mathbb{Z} \oplus H_1(V^\prime) \longrightarrow \mathbb{Z} \longrightarrow 0.$$
Since the alternate sum of the rank of finitely generated abelian groups in an exact sequence is zero, the result follows.
\end{proof}

\begin{lemma}\label{lemma:rk1}
Let $M$, $V$ and $V^\prime$ be as in remark \ref{remark:oriM}, where $V$ and $V^\prime$ are regular neighbourhoods $N(1,4)$ and $N(0,2,3)$ respectively. Let $\pi_1(M)=\mathbb{Z}$. Then the the fundamental group of $V^\prime$ is neither trivial nor $\mathbb{Z}_k$ for any $k$.
\end{lemma}
\begin{proof}
We have $M=V \cup V^\prime=(\mathbb{S}^1 \times \mathbb{B}^3) \cup V^\prime$ with $V \cap V^\prime = \partial V=\partial V^\prime$ from Remark \ref{remark:oriM}.  We can extend the spaces $V$ and $V^\prime$ by open simplices. Without loss of generality, we can assume that $V$ and $V^\prime$ in the given hypothesis are open. Let $i_1 : \pi_1(V \cap V^\prime)\rightarrow \pi_1(V)$ and $i_2: \pi_1(V \cap V^\prime)\rightarrow \pi_1(V^\prime)$ be the maps induced from inclusion maps $j_1: V\cap V^\prime \rightarrow V$ and $j_2: V\cap V^\prime \rightarrow V^\prime$ respectively. Since $V=\mathbb{S}^1 \times \mathbb{B}^3$ and $V\cap V^\prime = \mathbb{S}^1 \times \mathbb{S}^2$, if we let $\pi_1(V \cap V^\prime) = \langle \alpha \rangle$ then $\pi_1(V ) = \langle \alpha \rangle$ and $i_1(\alpha) = \alpha$.

If we assume to the contrary that $\pi_1(V^\prime)=\langle e \rangle \mbox{(or } \langle \beta | \beta^k \rangle)$ then Seifert-van Kampen Theorem implies $\pi_1(M)=\langle e \rangle \mbox{(or } \langle \beta | \beta^k \rangle)$ which is a contradiction as fundamental group of $M$ is $\mathbb{Z}$. Hence, the lemma  follows.
\end{proof}

\begin{lemma}\label{lemma:bound02}
Let $V$ and $V^\prime$ be as in remark \ref{remark:oriM}. Then, $0 \leq \beta_1(V^\prime) \leq 2.$
\end{lemma}

\begin{proof}
If $\beta_1(V^\prime)=k$ then $g_{14}=\beta_2(V^\prime)+4-k$ using Equation \eqref{eqn:beta21} and Lemma \ref{lemma:ori} for orientable case. Since each edge is a face of at least one triangle, the result follows.
\end{proof}

\begin{proposition}[\cite{rs72}]\label{prop:roursander}
Let $M$ be a manifold and $X \subset \mbox{int } M$ be a polyhedron. If $X$ collapses onto $Y$ then a regular neighbourhood
of $X$ is PL-homeomorphic to a regular neighbourhood of $Y$.
\end{proposition}

\begin{theorem}\label{theorem:main}
	Let $M$ be a closed orientable $4$-manifold with fundamental group $\mathbb{Z}$ . Let $(\Gamma,\gamma)$ be a semi-simple crystallization representing $M$. Then, there exists a handle decomposition of $M$ as one of the following:
	\begin{enumerate}[$(1)$]
		\item  one $0$-handle, two $1$-handles, $1+\beta_2(M)$ $2$-handles, one $3$-handle and one $4$-handle,
		\item one $0$-handle, one $1$-handle, $\beta_2(M)$ $2$-handles, one $3$-handle and one $4$-handle.
	\end{enumerate}
\end{theorem}

\begin{proof} Let $(\Gamma,\gamma)$ be a semi-simple crystallization representing $M$. We write $M=N(1,4) \cup N(0,2,3)=V \cup V^\prime$, where $V=\mathbb{S}^1 \times \mathbb{B}^3$ by Remark \ref{remark:oriM}. Now, we have to analyse $V^\prime.$ For $i \geq 1$, let $A_i$ be the set of all  triangles which have same boundary. Then, for $i \neq j$, some edges of a triangle in $A_i$ are not identical with the edges of triangle in $A_j$. Since the number of edges with  same labeled end vertices is $2$, there exist $2^3$ triangles such that none of them share the same boundary. Thus, we have $1 \leq i \leq  8.$ Let $k_i$ be the cardinality of $A_i$ for each $i$. Let $B_i$ denote the set of triangles obtained by contracting one triangle from $A_i$. Then $|B_i|=k_i-1$.
	
	Consider $\mathcal{K}(0,2,3)$. Now, each $A_i (1 \leq i \leq 8)$ may not be there in $\mathcal{K}(0,2,3)$.
	Let $A_{j_1}, A_{j_2},\dots, A_{j_q}$ be the subcollection of $\{A_i : 1 \leq i \leq 8 \}$ composing the cell complex $\mathcal{K}(0,2,3)$. In other words, $\mathcal{K}(0,2,3)=\cup_{r=1}^{q} A_{j_r}.$
	Now, $m+1$ number of triangles with same boundary contribute $m$ number of $2$-dimensional holes. This implies that $2$-holes in $\mathcal{K}(0,2,3)$ should be at least the sum of all the $2$-holes in $A_{j_r}(1 \leq r \leq q)$. That is,
	\begin{equation*}
		\Big( \sum_{r=1} ^q k_{j_r} \Big) - q \leq \beta_2(\mathcal{K}(0,2,3))=\beta_2 (V^\prime).
	\end{equation*}
	The last equality follows because $V^\prime$ being regular neighbourhood of $\mathcal{K}(0,2,3)$ contracts onto $\mathcal{K}(0,2,3)$. We also note that $2 \leq q \leq 8$ as each edge must be a face of at least one triangle.\\
	\noindent \textbf{Case A.} Let us first consider
	\begin{align}\label{eqn:sum}
		\Big( \sum_{r=1} ^q k_{j_r} \Big) - q = \beta_2 (V^\prime).
	\end{align}
	In this case, $\mathcal{K}(0,2,3)$ consists of such a subcollection $A_{j_r}(1 \leq r \leq q)$, where only the triangles with same boundary contribute to $\beta_2(\mathcal{K}(0,2,3))(=\beta_2(V^\prime))$. Equivalently, If we choose one triangle $T_{j_r}$ from each $A_{j_r}$  then  $\beta_2 (\cup_r T_{j_r})=0$.  It follows from Lemma \ref{lemma:bound02} that $\beta_1(V^\prime)=0,1 \mbox{ or } 2$.
	
	\noindent \textbf{Case 1.} Suppose $\beta_1(V^\prime)=2$. 
	By the proof  of Lemma \ref{lemma:bound02}, we have $g_{14}=\beta_2(V^\prime)+2$. Using Equation \eqref{eqn:sum}, we get $q=2$. This implies that any triangle in $A_{j_1}$ does not have any common edge with any triangle in $A_{j_2}$ in $\mathcal{K}(0,2,3)$ because each edge must be a face of at least one triangle.
	By contracting two triangles one from each $A_{j_1}$ and $A_{j_2}$, we observe that  $|\mathcal{K}(0,2,3)|$ contracts onto a  CW complex $K^\prime$; with single $0$-cell, two $1$-cells and $2$-cells consisting of sets $B_{j_1}$ and $B_{j_2}$ with cardinality $k_{j_1}-1$ and $k_{j_2}-1$ respectively.
	
	Now, a small regular neighbourhood of single $0$-cell in geometric carrier of $K^\prime$ is a $0$-handle $H^{(0)}=\mathbb{D}^4$, two $1$- cells contribute two $1$-handles and $2$-cells give $\beta_2(V^\prime)$ number of $2$-handles by Equation \eqref{eqn:sum}. Thus, by using Proposition \ref{prop:roursander},
	$$V^\prime= H^{(0)}\cup \left(H^{(1)}_1 \cup H^{(1)}_2 \right) \cup \left(H^{(2)}_1\cup \dots \cup H^{(2)}_{\beta_2(V^\prime)} \right).$$
	Now, the boundary identification between $V$ and $V^\prime$ is attachment of one $3$- and one $4$-handle, that is done uniquely by Proposition \ref{prop:unique} in \cite{a79}. Further, $\beta_2(V^\prime)$ equals $1+\beta_2(M)$ from Equation \eqref{eqn:beta21}. Thus,
	$$M=H^{(0)}\cup \left(H^{(1)}_1 \cup H^{(1)}_2 \right) \cup \left(H^{(2)}_1\cup \dots \cup H^{(2)}_{1+\beta_2(M)} \right) \cup H^{(3)} \cup H^{(4)}.$$
	
	\noindent \textbf{Case 2.} Suppose $\beta_1(V^\prime)=1$. By the proof  of Lemma \ref{lemma:bound02}, we have $g_{14}=\beta_2(V^\prime)+3$. Using Equation \eqref{eqn:sum}, we get $q=3$. 
	As we approached in Case (1), we contract three triangles one from each $A_{j_r}, r \in \{1,2,3 \}$ and observe that  $|\mathcal{K}(0,2,3)|$ contracts onto a  CW complex $K^\prime$; with one $0$-cell, one $1$-cell and $2$-cells consisting of sets $B_{j_r}$ with cardinality $k_{j_r}-1$, $r \in \{1,2,3\}$.

	Now, a small regular neighbourhood of single $0$-cell in the geometric carrier of $K^\prime$ is a $0$-handle $H^{(0)}=\mathbb{D}^4$, one $1$- cell contributes one $1$-handle and $2$-cells give $\beta_2(V^\prime)$ number of $2$-handles by Equation \eqref{eqn:sum}. Thus, by using Proposition \ref{prop:roursander}, 
	 $$V^\prime= H^{(0)}\cup H^{(1)}_1  \cup \left(H^{(2)}_1\cup \dots \cup H^{(2)}_{\beta_2(V^\prime)} \right).$$ Now, the boundary identification between $V$ and $V^\prime$ is attachment of one $3$- and one $4$-handle and using $\beta_2(V^\prime)$ equals $\beta_2(M)$ from Equation \eqref{eqn:beta21}. We have, 
	$$M=H^{(0)}\cup H^{(1)}_1  \cup \left(H^{(2)}_1\cup \dots \cup H^{(2)}_{\beta_2(M)} \right) \cup H^{(3)} \cup  H^{(4)} .$$
	
	\noindent \textbf{Case 3.} Suppose $\beta_1(V^\prime)=0$. Then we have $g_{14}=\beta_2(V^\prime)+4$. Using Equation \eqref{eqn:sum}, we get $q=4$. 
	We claim that the fundamental group of $|\mathcal{K}(0,2,3)|=\{A_{i_r}|1 \leq r \leq 4\}$ is $ \mathbb{Z}_2$ or $\langle e \rangle$. 
	To find its fundamental group, it is sufficient to assume that $|A_{i_r}|=1$.
	Let $l_i$, $r_i$ and $d_i$ be  the edges with end points $\{0,2\}$, $\{0,3\}$ and $\{2,3\}$ respectively, for $i \in \{1,2\}$. If the number of triangles passing through any of the  $l_1,l_2,r_1,r_2,d_1,d_2$ edges is two then the space is either a projective plane or the space in Figure \ref{fig:figure1(chapter5)}. Else, without loss of generality, we assume that the number of triangles passing through $l_1$ is $3$. Then there is exactly one triangle with edge $l_2$. Now, contracting the triangle from the free side $l_2$, we reduce the number of triangles to three such that one of $r_1,r_2,d_1,d_2$ edges becomes free. Again, we contract the triangle from the free edge, and we get the resultant space to be contractible.  However, by the equivalent statement for Case $A$, we note that the space in Figure \ref{fig:figure1(chapter5)} does not include in the Case $A$  as $\beta_2(\cup_r A_{j_r}) \neq 0$.  Hence, the claim follows.
	
	Now, If the fundamental group of $K(0,2,3)$ is $\mathbb{Z}_2$ or $\langle e \rangle$, then the fundamental group of $V^\prime$ is $\mathbb{Z}_2$ or $\langle e \rangle$, which is a contradiction by Lemma \ref{lemma:rk1}. Thus, this case is not possible.
	
	\smallskip
	
	\noindent \textbf{Case B.} Let $\Big( \sum_{r=1} ^q k_{j_r} \Big) - q < \beta_2 (V^\prime).$ In other words, this case considers the subcollection $\{A_{j_r}: 1 \leq r \leq q\}$  such that the triangles with different boundary also contributes to $\beta_2(V^\prime)$. In other words, the subcollection consists of the triangles such that different $A_i$'s also contribute to $\beta_2(V^\prime)$. Now, to find such a subcollection explicitly, it is enough to work with the condition that $|A_i|=1$ for each $i$. 
	So, we have to find the subsets of all $2^3$ possible triangles which give non-zero second Betti number. Let $T= \cup_{r=1}^{q} A_{j_r}$, where $|A_{j_r}|=1$ for each $r$. 
	
	If $5 \leq q \leq 8$ then we claim that the fundamental group of $T$ is trivial. 
	For, we let $l_i$, $r_i$ and $d_i$ be the the edges with end points $\{0,2\}$, $\{0,3\}$ and $\{2,3\}$ respectively, for $i \in \{1,2\}$. It follows from the Pigeonhole principle that the loops $l_1l_2$, $r_1r_2$ or $d_1d_2$ are contractible. Therefore, the fundamental group of $T$ is trivial, which  contradicts the Lemma \ref{lemma:rk1}. If $q <4$ then $\beta_2(T)$ is zero by finding second homology group manually and thus this possibility is not included in Case $B$. If $q=4$ then the non-zero $\beta_2(T)$ is given by only the space in Figure \ref{fig:figure1(chapter5)} consisting of four triangles, say $A_{i_1}, A_{i_2},A_{i_3}$ and $A_{i_4}$, which is PL-homeomorphic to a CW-complex consisting of one $l$-cell, for each $l$, $0 \leq l \leq 2$. Thus, there exists only one subcollection in Case $B$. Now to write a handle decomposition we work in the general case, that is, we do not restrict ourselves to the condition with $|A_i|=1.$
	
	Now, we contract four triangles one from each $A_{i_t}, t \in \{1,2,3,4 \}$ and get that $|\mathcal{K}(0,2,3)|$ contracts to a CW complex $K^\prime$; with one $0$-cell, one $1$-cell and $2$-cells with cardinality one more the cardinality of union of sets $\{B_{i_r}|1 \leq r \leq 4\}$.
	
	Now, small regular neighbourhood of single $0$-cell in geometric carrier of $K^\prime$ is a $0$-handle $H^{(0)}=\mathbb{D}^4$, one $1$- cell contribute one $1$-handle and $2$-cells give $\beta_2(V^\prime)$ number of $2$-handles. Thus, $$V^\prime= H^{(0)}\cup H^{(1)}_1  \cup \left(H^{(2)}_1\cup \dots \cup H^{(2)}_{\beta_2(V^\prime)} \right).$$ Now, the boundary identification between $V$ and $V^\prime$ is attachment of one $3$- and one $4$-handle that is done in a unique way by Proposition \ref{prop:unique} in \cite{a79}, and using $\beta_2(V^\prime)$ equals $\beta_2(M)$ from Equation \eqref{eqn:beta21}. Thus,
	$$M=H^{(0)}\cup H^{(1)}_1  \cup \left(H^{(2)}_1\cup \dots \cup H^{(2)}_{\beta_2(M)} \right) \cup H^{(3)} \cup  H^{(4)} .$$
	This completes the proof.

	\begin{figure}[ht]
		\tikzstyle{ver}=[]
		\tikzstyle{vert}=[circle, draw, fill=black!100, inner sep=0pt, minimum width=4pt]
		\tikzstyle{vertex}=[circle, draw, fill=black!00, inner sep=0pt, minimum width=4pt]
		\tikzstyle{edge} = [draw,thick,-]
		\centering
		
		\begin{tikzpicture}[scale=0.5]

			\begin{scope}[shift={(8,0)}]
				\foreach \x/\y in {90/A_{i_{1}},270/A_{i_{3}},180/A_{i_{2}},0/A_{i_{4}}}{
					\node[ver] (\y) at (\x:1.3){\tiny{$\y$}};
				}

				\foreach \x/\y in {135/v_{3},315/v_{1}}{
					\node[vertex] (\y) at (\x:3){0};
				}
				\foreach \x/\y in {45/v_{2},225/v_{4}}{
					\node[vertex] (\y) at (\x:3){3};
				}
				\node[vertex] (v) at (0,0){2};
				\foreach \x/\y in {v_{1}/v_{2},v_{2}/v_{3},v_{3}/v_{4},v_{4}/v_{1},v_{1}/v,v_{2}/v,v_{3}/v,v_{4}/v}{
					\path[edge] (\x) -- (\y);}
				
				\node[ver] () at (0,2.05){$\rightarrow$};
				\node[ver] () at (-2.1,0){$\downarrow$};
				\node[ver] () at (2.15,0){$\uparrow$};
				\node[ver] () at (2.15,0.2){$\uparrow$};
				\node[ver] () at (0,-2.2){$\leftarrow$};
				\node[ver] () at (0.2,-2.2){$\leftarrow$};
				
			\end{scope}

		\end{tikzpicture}
		\caption{} \label{fig:figure1(chapter5)}
	\end{figure}

\end{proof}

\noindent The following is the obvious question to be asked next:

\noindent \textbf{Question:} What can be said about the generalisation of the result with no restriction on the fundamental group of PL $4$ manifold?

   \bigskip

\noindent {\bf Acknowledgement:} 
The first author is supported by Science and Engineering Research Board (CRG/2021/000859).

{\footnotesize

\end{document}